\documentclass{amsart}[12pt]
\usepackage{amsmath, amsthm, amscd, amsfonts,}
\usepackage{graphicx,float}

\usepackage{amssymb}

\usepackage{pb-diagram}
\newtheorem{theorem}{Theorem}[section]
\newtheorem{lemma}[theorem]{Lemma}

\newtheorem{corollary}[theorem]{Corollary}

\newtheorem{definition}[theorem]{Definition}

\newtheorem{remark}[theorem]{Remark}

\numberwithin{equation}{section}

\def\rmark{\mbox{$\rm\bf\rule{0.06em}{1.45ex}\kern-0.05em R$}}
\def\pmark{\mbox{$\rm\bf\rule{0.06em}{1.45ex}\kern-0.05em P$}}
\def\nmark{\mbox{$\rm\bf\rule{0.06em}{1.45ex}\kern-0.05em N$}}
\def\vdash{\mbox{$\rm\| \kern-0.13em -$}}

\usepackage{pb-diagram}

\def\rmark{\mbox{$\rm\bf\rule{0.06em}{1.45ex}\kern-0.05em R$}}
\def\pmark{\mbox{$\rm\bf\rule{0.06em}{1.45ex}\kern-0.05em P$}}
\def\nmark{\mbox{$\rm\bf\rule{0.06em}{1.45ex}\kern-0.05em N$}}
\def\vdash{\mbox{$\rm\| \kern-0.13em -$}}

\begin{document}

\title[On certain maximality principles]{On certain maximality principles}

\author[R. Mohammadpour]{Rahman Mohammadpour}

\thanks{This paper is based on  chapter 3 of the author's master of science thesis that is written under the supervisions of Mohammad Golshani and Massoud Pourmahdian.}
\thanks{The author is thankful to Mohammad Golshani, for all his help and support during the author's master of science degree and for giving useful comments on the draft version of this paper. The author is also grateful to Ali Enayat for his encouragement and his helpful comments on the draft version of the paper.}
 \maketitle

\begin{abstract}
We present new, streamlined proofs of certain maximality principles studied by Hamkins
and Woodin. Moreover, we formulate an intermediate maximality principle, which is
shown here to be equiconsistent with the existence of a weakly compact cardinal $\kappa$ such that $V_{\kappa}\prec V$.
\end{abstract}

\maketitle
\section{Introduction}
In his paper ``A Simple Maximality Principle'' [3], Joel Hamkins introduced a maximality principle that is exactly axiom $S5$ in the sense of modal logic. He also considered various versions of his maximality principle and proved the consistency of some of them, but the status of some others remained open. The necessary maximality principle ($\square {\rm MP}(\mathbb{R})$) is a strong principle that apparently its consistency is proved by Hugh Woodin in an unpublished work [3]. Also, in a joint work [4], Hamkins and Woodin proved that $\square_{c.c.c.}{\rm MP}_{c.c.c.}(\mathbb{R})$ is equiconsistent with the existence of a weakly compact cardinal. In this paper, we present new proofs of some results of Hamkins [3] which are simpler than the original ones. Moreover we introduce $\square_{c.c.c.}{\rm MP}(\mathbb{R})$ and  show that it is equiconsistent with the existence of a weakly compact cardinal $\kappa$ such that $V_{\kappa}\prec V$.
\section{Hamkins' Maximality principle}
For a sentence $\phi$ in the language of set theory and a model $V$ of $ZFC$, we say $\phi$ is forceable ($\diamondsuit\phi$) over $V$, if there exists a forcing notion $\mathbb{P}\in V$ such that $V^{\mathbb{P}}\models\phi$ and $\phi$ is necessary ($\square\phi$) over $V$, if for all forcing notions $\mathbb{P}\in V$, $V^{\mathbb{P}}\models\phi$.
When we restrict ourselves to the class of $c.c.c.$ forcing notions,  then we say $\phi$ is $c.c.c.$-forceable $(\diamondsuit_{c.c.c.}\phi)$ or $c.c.c.$-necessary $(\square_{c.c.c.}\phi)$. Note that $\diamondsuit\phi$, $\square\phi$, $\diamondsuit_{c.c.c.}\phi$ and $\square_{c.c.c.}\phi$ are all first order definable.
Now we define some principles that are schemes in the first order language.
\begin{definition}{(The Maximality Principle)}
 Hamkins' maximality principle $({\rm MP}(X))$ asserts
 $$\diamondsuit\square\phi\Longrightarrow\phi,$$
 whenever $\phi$ is a sentence in the language of set theory with parameters from $X$.
\end{definition}
\begin{definition}{(The Necessary Maximality Principle)}
Hamkins' necessary maximality principle $(\square {\rm MP}(X))$ asserts  ${\rm MP}(X)$ is necessary, i.e., for any forcing notion $\mathbb{P}$, $V^{\mathbb{P}}\models {\rm MP}(X)$; in particular $V\models{\rm MP}(X).$
\end{definition}
We use the notations ${\rm MP}_{c.c.c}(X)$ and $\square_{c.c.c.}{\rm MP}_{c.c.c.}(X)$, when we restrict ourselves to the class of $c.c.c.$-forcing notions. We may note that Hamkins has used the notation $\square {\rm MP}_{c.c.c.}(X)$ to denote our $\square_{c.c.c.}{\rm MP}_{c.c.c.}(X)$.

Let $V_{\delta}\prec V$ be the scheme that asserts $\forall x\in V_{\delta}(\phi(x)\longleftrightarrow \phi(x)^{V_{\delta}})$, whenever $\phi(x)$ is a formula with only one free variable.
The following is an immediate consequence of the Montague-Levy reflection theorem.
\begin{lemma}{([3])}
If $ZFC$ is consistent, then is $ZFC+V_{\delta}\prec V$.
\end{lemma}
The first item of following lemma is a special case of theorem 14.1 of [1] and two other items easily follow from the first one.
\begin{lemma}{([1], Theorem 14.1)}
\begin{enumerate}
\item
If $\mathbb{P}$ is a forcing notion of size $\kappa$ and $\Vdash_{\mathbb{P}}|\kappa|=\aleph_{0}$, then $\mathbb{P}$ is forcing equivalent to $Coll(\omega,\kappa)$.
\item
If $\kappa$ is an inaccessible cardinal and $\mathbb{P}\in V_{\kappa}$, then $V^{\mathbb{P}}\subseteq V^{Coll(\omega,<\kappa)}$.
\item
If $\mathbb{P}$ is a forcing notion of size  $\leq \kappa$, then there exists a projection from $Coll(\omega,\kappa)$ to $\mathbb{P}$.
\end{enumerate}
\end{lemma}

\begin{lemma}
Suppose that $\mathbb{P}\subseteq V_{\delta}$ is a $\delta-c.c$ forcing notion. If $G$ is a $\mathbb{P}$-generic filter over $V$ and $V_{\delta}\prec V$, then $G$ is $\mathbb{P}$-generic over $V_{\delta}$ and $V_{\delta}[G]\prec V[G] $.
\end{lemma}
\begin{proof}
Let $A\subseteq\mathbb{P}$ be a maximal antichain in $V_{\delta}$, hence $A\in V_{\delta}$ and it is a maximal antichain in $V$, which implies $G\cap A\neq\varnothing$, so $G$ is $\mathbb{P}$-generic over $V_{\delta}$. Suppose that $V[G]\models\exists x\phi(x)$, then there exists $p\in G$ such that $V\models p\Vdash\exists x\phi(x)$, thus by elementarity, $V_{\delta}\models\exists x p\Vdash\phi(x)$. Then using $\delta-c.c$-ness, one can find $x\in V_{\delta}[G]$ such that $V_{\delta}[G]\models\phi(x)$. Hence $V_{\delta}[G]\prec V[G]$.
\end{proof}
The following lemma is proved by Hamkins and Woodin using the method of Boolean valued models; we give a proof which avoids the use of Boolean valued models.
\begin{lemma}{([4])}
Let $\kappa$ be a weakly compact cardinal and let $\mathbb{P}$ be a $\kappa-c.c.$ forcing notion. Then, for any $x\in H(\kappa)^{\mathbb{P}}$,  there exists $\mathbb{Q}\lhd\mathbb{P}$ such that $x\in H(\kappa)^{\mathbb{Q}}$ and $\mathbb{Q}$ is a forcing notion of size less than $\kappa$.
\end{lemma}
\begin{proof}
If $|\mathbb{P}|<\kappa$, then there is nothing to prove. So, assume $|\mathbb{P}|\geq\kappa$, and choose $\theta>\kappa$ large enough regular such that $\mathbb{P}\in H_{\theta}$. Without loss of generality suppose that $x\subseteq\kappa$ and $|x|<\kappa$. Let $\tau$ be a $\mathbb{P}$-name for $x$ and find an elementary substructure $\mathcal{M}$ of $H_{\theta}$ such that:
\begin{itemize}
\item
$\kappa,\tau,\mathbb{P}\in\mathcal{M}$,
\item
$|\mathcal{M}|=\kappa$,
\item
$\mathcal{M}^{<\kappa}\subseteq\mathcal{M}$.
\end{itemize}
The existence of $\mathcal{M}$ is guaranteed by the inaccessibility of $\kappa$. We divide the proof into two cases:
\begin{enumerate}
\item[Case 1.]
$|\mathbb{P}|=\kappa$: Since $\kappa$ is weakly compact, there exists a transitive model $\mathcal{N}$ and an elementary embedding $j:\mathcal{M}\longrightarrow\mathcal{N}$ with $crit(j)=\kappa$. We show that
$$N\models(\exists\mathbb{Q}\lhd j(\mathbb{P}))[|\mathbb{Q}|<| j(\mathbb{P})|\wedge(\exists\sigma 1_{j(\mathbb{P})}\Vdash\sigma=j(\tau)].$$
Without loss of generality suppose that $\mathbb{P}\subseteq V_{\kappa},$ which implies $\mathbb{P}\subseteq j(\mathbb{P})$. Let $A\subseteq\mathbb{P}$ be a maximal antichain in $\mathcal{N}$. So $|A|<\kappa,$ hence $A\in \mathcal{M}$ and
$$\mathcal{M}\models\forall p\in\mathbb{P}\exists q\in A~q\parallel p.$$
Using $j,$ we have
$$N\models\forall p\in j(\mathbb{P})\exists q\in j(A)=A~q\parallel p.$$
On the other hand, since $x\in H(\kappa)^{\mathbb{P}}$, and $\mathbb{P}$ is $\kappa-c.c.,$ we can assume $|\tau|<\kappa$, so $j(\tau)=\tau$. Put $\sigma=\tau$ which implies $1_{j(\mathbb{P})}\Vdash\sigma=\tau.$ Thus
$$N\models(\exists\mathbb{Q}\lhd j(\mathbb{P}))[|\mathbb{Q}|<| j(\mathbb{P})|\wedge(\exists\sigma 1_{j(\mathbb{P})}\Vdash\sigma=j(\tau)].$$
Now by elementarity of $j$, we have in $\mathcal{M}$:
$$\exists\mathbb{Q}\lhd \mathbb{P}~|\mathbb{Q}|<| \mathbb{P}|\wedge\exists\sigma 1_{\mathbb{P}}\Vdash\sigma=\tau.$$

\item[Case 2.]$|\mathbb{P}|>\kappa$: Consider $\mathbb{P}\cap\mathcal{M}$ and let $A\subseteq\mathbb{P}\cap\mathcal{M}$ be a maximal antichain. Then $|A|<\kappa$,  hence $A\in \mathcal{M}$ and $A$ is a maximal antichain in $\mathbb{P}$. Hence $\mathbb{P}\cap\mathcal{M}\lhd\mathbb{P}$. Then by case 1, there exists $\mathbb{Q}$ of size less than $\kappa$ and a $\mathbb{Q}$-name $\sigma$ such that $\mathbb{Q}\lhd\mathbb{P}\cap\mathcal{M}\lhd\mathbb{P}$ and $\vdash_{\mathbb{Q}}\tau=\sigma$.
\end{enumerate}
\end{proof}
\begin{theorem}{(Hamkins)}
The consistency of $ZFC$ implies the consistency of $ZFC+{\rm MP}$.
\end{theorem}
\begin{proof}
Suppose that $V_{\delta}\prec V$. Let $\mathbb{P}=Coll(\omega,\delta)$ and let $G$ be $\mathbb{P}$-generic over $V$. If $V[G]\models\diamondsuit\square\phi$, then for some $p\in G$ 
$$p\Vdash\exists\mathbb{Q}(\exists q\in\mathbb{Q}~q\Vdash\square\phi).$$
Since $p\in V_{\delta}$ and $V_{\delta}\prec V$, one can find $\dot{\mathbb{Q}}$ and $\dot{q}$ in $V_{\delta}$, such that $(p,\dot{q})\Vdash\square\phi$, so $V^{\mathbb{P}\ast\dot{\mathbb{Q}}}\models\square\phi$. On the other hand by lemma 2.4 and the fact that $|\mathbb{P}\ast\dot{\mathbb{Q}}|=|\delta|$, we have $\mathbb{P}\ast\dot{\mathbb{Q}}\approx\mathbb{P}$ which implies $V^{\mathbb{P}}\models\square\phi$. Thus $V^{\mathbb{P}}\models\phi,$ that completes the proof.
\end{proof}
\begin{theorem}{(Hamkins)}
$ZFC+{\rm MP}(\mathbb{R})$ is equiconsistent with $ZFC$ plus the existence of an inaccessible cardinal $\kappa$ such that $V_{\kappa}\prec V$.
\end{theorem}
\begin{proof}
(Right to Left) Assume $\kappa$ is an inaccessible cardinal such that $V_{\kappa}\prec V$. Consider $\mathbb{P}=Coll(\omega,<\kappa)$. Let $G$ be a $\mathbb{P}$-generic forcing over $V$. Suppose that $\bar{x}\in2^{\omega}\cap V[G]$, so $\bar{x}\in V[G_{\lambda}]$, for some regular cardinal $\lambda<\kappa$, where $G_{\lambda}=G\cap Coll(\omega,<\lambda)$ is generic for $\mathbb{P}_{\lambda}=Coll(\omega,<\lambda)$. Since $V_{\kappa}\prec V,$ by lemma 2.5, $V_{\kappa}[G_{\lambda}]\prec V[G_{\lambda}]$. It follows that there exists $\mathbb{Q}\in V_{\kappa}[G_{\lambda}]$  such that $V[G_{\lambda}]^{\mathbb{Q}}\models\square\phi(\bar{x})$, hence $V^{\mathbb{P}_{\lambda}\ast\dot{\mathbb{Q}}}\models\phi(\bar{x})$. Then $V[G]\models\phi(\bar{x})$, by lemma 2.4 (3).\\
The other side has\ appeared as lemma 3.2 in [3]. We give a proof for completness.\\
(Left to Right) Assume ${\rm MP}(\mathbb{R})$. First we claim that $\omega_1$ is inaccessible to the reals. Thus let $r$ be a real and let $\phi(r)$ be the statement ``the $\omega_1$ of $L[r]$ is countable'', that is forceably necessary. Thus it is true by $MP(\mathbb{R})$, so we have $\omega_1^{L[r]}<\omega_1$. This implies  $\delta=\omega_{1}$ is inaccessible to the reals, and hence it is inaccessible in $L$. Now it is enough to show that $L_{\delta}\prec L$. Suppose that $L\models\exists x\phi(x,a)$, where $\phi(x,a)$ has parameters from $L_{\delta}$, that are coded in a single real $a$. Let $\psi$ be the following statement
\begin{center}
 ``The least $\alpha$ such that there is an $x\in L_{\alpha}$ with $\phi(x,a)^{L_{\alpha}}$, is countable''
\end{center}
It is forceably necessary, thus it's already true, that means the least such $\alpha$ is countable, so there exists $y\in L_{\delta}$ with $\phi(y,a)$.

\end{proof}

The following theorem is proved by Hamkins and Woodin [4]:
\begin{theorem}{(Hamkins-Woodin [4])}
$ZFC+\square_{c.c.c}{\rm MP}_{c.c.c}(\mathbb{R})$ is equiconsistent with $ZFC$ plus the existence of a weakly compact cardinal.
\end{theorem}
We now state and prove a generalization of the above theorem.
The next result is obtained in joint work with M. Golshani, and is presented here with his kind permission. 
\begin{theorem}
The followings are equiconsistent:
\begin{enumerate}
\item
$ZFC$ plus the existence of a weakly compact cardinal $\kappa$ such that $V_{\kappa}\prec V$.
\item
$ZFC+\square_{c.c.c}{\rm MP}(\mathbb{R})$.
\end{enumerate}
\end{theorem}
\begin{remark}
Clearly, $\square_{c.c.c}{\rm MP}(\mathbb{R})$  is stronger than $\square_{c.c.c}{\rm MP}_{c.c.c}(\mathbb{R}),$ but weaker than $\square {\rm MP}(\mathbb{R})$. It is easily seen that the condition 1 of Theorem 2.10 is strictly stronger than the existence of a weakly compact cardinal, but it is consistent to have this condition,  if, for example, one assumes the existence of a measurable cardinal.
\end{remark}
\begin{proof}
($1\Longrightarrow 2$) Let $\kappa$ be a weakly compact cardinal such that $V_{\kappa}\prec V$. Let $\mathbb{P}=Coll(\omega,<\kappa)$ and let $G$ be a $\mathbb{P}$-generic filter over $V$. By Theorem 2.8,  $V[G]\models {\rm MP}(\mathbb{R})$. Now let $\mathbb{Q}$ be an arbitrary $c.c.c.$ forcing in $V[G]$ and let $H$ be a $\mathbb{Q}$-generic filter over $V[G]$. Our aim is to show that $V[G][H]\models {\rm MP}(\mathbb{R})$. Thus assume $r\in\mathbb{R}\cap V[G][H]$ and $V[G][H]\models\diamondsuit\square\phi(r)$.

Clearly $\mathbb{P}\ast\dot{\mathbb{Q}}$ is $\kappa-c.c$., so Lemma 2.6 guarantees the existence of a $\kappa-c.c.$ forcing notion $\mathbb{S}\lhd\mathbb{P}\ast\dot{\mathbb{Q}}$ such that $|\mathbb{S}|<\kappa$ and $r\in V^{\mathbb{S}}$. Assume, without loss of generality, that  $\mathbb{S}\in V_{\kappa}$.

Let $K$ be an $\mathbb{S}$-generic filter over $V$, so that
$$V[K]\subseteq V[G\ast H].$$

Since $V[G\ast H]\models\diamondsuit\square\phi(r)$, we have $V[K]\models\diamondsuit\square\phi(r)$. On the other hand we have $V_{\kappa}[K]\prec V[K]$, so $V_{\kappa}[K]\models\diamondsuit\square\phi(r)$. Thus there exists $\mathbb{T}\in V_{\kappa}[K]$ and a $\mathbb{T}$-generic filter $L$ such that $V_{\kappa}[K][L]\models\square\phi(r)$. But we have a canonical projection $\pi:Coll(\omega,<\kappa)\longrightarrow \mathbb{S}\ast\dot{\mathbb{T}}$, which implies $V[G\ast H]\models\phi(r)$.

($2\Longrightarrow1$)
Assume $V\models \square_{c.c.c}{\rm MP}(\mathbb{R})$, and let $\mathbb{P}$ be some $c.c.c$ forcing notion which forces $MA+ \neg CH$. Then $V^{\mathbb{P}}\models{\rm MP}(\mathbb{R})+\square_{c.c.c}{\rm MP}_{c.c.c}(\mathbb{R})$. It follows from Theorem 2.8  that $\delta=\omega_1^{V^{\mathbb{P}}}$ is inaccessible in $L$ and $L_{\delta}\prec L$. Since $MA$ holds, it follows from a theorem of Harrington-Shelah [5] that $\omega_1$ is weakly compact in $L$.
\end{proof}
As a corollary, we obtain a simpler proof of Hamkins-Woodin's theorem 3.1.
\begin{corollary}
Suppose there exists a weakly compact cardinal $\kappa$ such that $V_{\kappa}\prec V$. Then $\square_{c.c.c}{\rm MP}_{c.c.c}(\mathbb{R})$ is consistent.
\end{corollary}

Department of mathematics and Computer Sciences, Amirkabir University of Technology,
Hafez avenue 15194, Tehran, Iran.

E-mail address: rahmanmohammadpour@gmail.com

\end{document}